\documentclass[letterpaper,11pt,oneside,reqno]{amsart}
\usepackage{amsfonts,amsmath, amssymb,amsthm,amscd}
\usepackage[usenames,dvipsnames]{color}
\usepackage{bm}
\usepackage{mathrsfs}
\usepackage{yfonts}
\usepackage[mpexclude,DIV13]{typearea}
\usepackage{verbatim}
\usepackage{hyperref}
\usepackage{graphicx}
\usepackage[latin1]{inputenc}
\usepackage{latexsym}
\usepackage{lscape}
\usepackage{epsfig}
\usepackage{subfigure}

\numberwithin{equation}{subsection}
\usepackage{setspace}
\include{bibtex}

\begin{document}

\bibliographystyle{alpha}
\newcommand{\e}[0]{\epsilon}
\newcommand{\EE}{\ensuremath{\mathbb{E}}}
\newcommand{\PP}{\ensuremath{\mathbb{P}}}
\newcommand{\R}{\ensuremath{\mathbb{R}}}
\newcommand{\Rplus}{\ensuremath{\mathbb{R}_{+}}}
\newcommand{\C}{\ensuremath{\mathbb{C}}}
\newcommand{\Z}{\ensuremath{\mathbb{Z}}}
\newcommand{\N}{\ensuremath{\mathbb{N}}}
\newcommand{\Q}{\ensuremath{\mathbb{Q}}}
\newcommand{\Weyl}[0]{\ensuremath{\mathbb{W}}}
\newcommand{\YY}[0]{\ensuremath{\mathbb{Y}}}
\newcommand{\XX}[0]{\ensuremath{\mathbb{X}}}

\newcommand{\Real}{\ensuremath{\mathrm{Re}}}
\newcommand{\Imag}{\ensuremath{\mathrm{Im}}}
\newcommand{\re}{\ensuremath{\mathrm{Re}}}

\newcommand{\Sym}{\ensuremath{\mathrm{Sym}}}

\newcommand{\bfone}{\ensuremath{\mathbf{1}}}

\def \sgn {{\rm sgn}}
\newcommand{\var}{{\rm var}}

\newcommand{\Id}{\ensuremath{\mathrm{Id}}}
\renewcommand{\i}{\mathbf i}

\newtheorem{theorem}{Theorem}[section]
\newtheorem{partialtheorem}{Partial Theorem}[section]
\newtheorem{conj}[theorem]{Conjecture}
\newtheorem{lemma}[theorem]{Lemma}
\newtheorem{proposition}[theorem]{Proposition}
\newtheorem{corollary}[theorem]{Corollary}
\newtheorem{claim}[theorem]{Claim}
\newtheorem{formal}[theorem]{Critical point derivation}
\newtheorem{experiment}[theorem]{Experimental Result}
\newtheorem{prop}{Proposition}

\def\todo#1{\marginpar{\raggedright\footnotesize #1}}
\def\change#1{{\color{green}\todo{change}#1}}
\def\note#1{\textup{\textsf{\Large\color{blue}(#1)}}}

\theoremstyle{definition}
\newtheorem{remark}[theorem]{Remark}

\theoremstyle{definition}
\newtheorem{example}[theorem]{Example}

\theoremstyle{definition}
\newtheorem{definition}[theorem]{Definition}

\theoremstyle{definition}
\newtheorem{definitions}[theorem]{Definitions}

\begin{abstract}
We prove a intertwining relation (or Markov duality) between the $(q,\mu,\nu)$-Boson process and $(q,\mu,\nu)$-TASEP, two discrete time Markov chains introduced by Povolotsky in \cite{Pov}. Using this and a variant of the coordinate Bethe ansatz we compute nested contour integral formulas for expectations of a family of observables of the $(q,\mu,\nu)$-TASEP when started from step initial data. We then utilize these to prove a Fredholm determinant formula for distribution of the location of any given particle.
\end{abstract}

\title{The $(q,\mu,\nu)$-Boson process and $(q,\mu,\nu)$-TASEP}

\author[I. Corwin]{Ivan Corwin}
\address{I. Corwin, Columbia University,
Department of Mathematics,
2990 Broadway,
New York, NY 10027, USA,
and Clay Mathematics Institute, 10 Memorial Blvd. Suite 902, Providence, RI 02903, USA,
and Massachusetts Institute of Technology,
Department of Mathematics,
77 Massachusetts Avenue, Cambridge, MA 02139-4307, USA}
\email{ivan.corwin@gmail.com}

\maketitle
%
%\setcounter{tocdepth}{3}
%\tableofcontents
%\hypersetup{linktocpage}

\section{Introduction}\label{s.intro}

We study the $(q,\mu,\nu)$-Boson process and the $(q,\mu,\nu)$-TASEP, two discrete time interacting particle systems introduced by Povolotsky \cite{Pov}. The primary contribution of this paper is the discovery of an intertwining relationship between the Markov transition matrices for these processes. This generalizes the Markov duality previously discovered in \cite{BCS,BCdiscrete} for the $q$-Boson process and $q$-TASEP (special cases of the processes considered herein corresponding to setting $\nu=0$). Utilizing the coordinate Bethe ansatz solvability discovered in \cite{Pov} and following the strategy laid out in \cite{BCS,BCdiscrete} for the $\nu=0$ case, we prove nested contour integral formulas for expectations of a large class of observables of the $(q,\mu,\nu)$-TASEP, when started from step initial data. These formulas completely characterize the fixed time distribution of the process, and by following the approach of \cite{BorCor} (see also \cite{BCGS}) we extract two Fredholm determinant formulas for the one-point distribution of $(q,\mu,\nu)$-TASEP with step initial data.

There are a variety of interesting interacting particle systems which arise through limits or specializations of parameters of the processes considered herein. The Fredholm determinant formula which we prove here is likely to be amenable to asymptotic analysis (such as that performed in \cite{BorCor,BCF,BCFV, FV} for the special cases of $q$-TASEP, the O'Connell-Yor semi-discrete directed polymer and the Kardar-Parisi-Zhang equation).

This introductory section contains most of the main results of the paper, in addition to the notation and background necessary to state and prove them (in particular a brief introduction to $q$-hypergeometric series and Heine's 1847 $q$-generalization of Gauss' summation formula). Section \ref{s.genparsection} provides a general parameter extension of the processes considered in the introduction, and shows that the intertwining relationship and some of the Bethe ansatz solvability extends to this general setting. It also includes proofs of those results stated without proofs in the introduction. Section \ref{s.discuss} contains discussion of some related results in the literature, possible extensions to the present results and new directions of research.

\subsection{Notation}\label{s.notation}
We write $\Z_{>0} = \{1,2,\ldots\}$ and $\Z_{\geq 0} = \{0,1,\ldots\}$.  For $N\in \Z_{>0}$ define
$$\YY^N := (\Z_{\geq 0})^{\{0,1,\ldots,N\}} \qquad \textrm{and}\qquad \YY^N_k := \big\{\vec{y}\in \YY^N: \sum y_i = k\big\}.$$
For a vector $\vec{y}\in \YY^N$ and a vector $\vec{s}=(s_1,\ldots, s_N)$ with integers $s_i\in \{0,\ldots, y_i\}$ for all $i$, let
$$\vec{y}^{s_i}_{i,i-1} = \big(y_0,\ldots, y_{i-1}+s_i, y_i -s_i,\ldots, y_N\big).$$
%\vec{y}^{\vec{s}} &=& \big(y_0+s_1,y_1+s_2-s_1,y_2+s_3-s_2,\ldots, y_{N-1}+s_N-s_{N-1}, y_N-s_N\big).

Define the Weyl chamber
$$\Weyl^N_k := \big\{\vec{n} = (n_1,\ldots, n_k): N\geq n_1\geq \cdots \geq n_k \geq 0\big\}.$$
To every $\vec{y}\in \YY^N_k$ we may associate the ordered list of locations $\vec{n}(\vec{y})\in \Weyl^N_k$ such that for each $i\in \{0,\ldots, N\}$ the number of elements of $\vec{n}$ equal to $i$ is exactly $y_i$. Likewise to $\vec{n}\in \Weyl^N_k$ we associate the uniques $\vec{y}=\vec{y}(\vec{n})\in \YY^N_k$ such that $\vec{n}(\vec{y}) = \vec{n}$. For example for $N\geq 2$, if $\vec{y}$ was such that $y_0=3$, $y_1=1$, $y_2=1$ and all other $y_i=0$, then $\vec{y}$ maps to $\vec{n}(\vec{y}) = (2,1,0,0,0)$. For $I\subseteq \{1,\ldots, k\}$ let $\vec{n}^{-}_{I}$ denote the the vector $\vec{n}$ with $n_i$ replaced by $n_i-1$ for all $i\in I$.

For $N\in \Z_{>0}$ define
$$\XX^N := \big\{\vec{x} = (x_0,x_1,\ldots,X_N) \subset \Z^N: +\infty=x_0>x_1>\cdots>x_N\big\}.$$
For a vector $\vec{x}\in \XX^N$ and a vector $\vec{j}=(j_1,\ldots, j_N)$ with integers $j_i\in \{0,\ldots, x_{i-1}-x_{i}-1\}$ for all $i$, let
$$\vec{x}^{j_i}_{i} := \big(x_1,\ldots, x_{i}+j_i, \ldots, x_N\big).$$
%\vec{x}^{\vec{s}} &=& \big(x_1+s_1,\ldots, x_N+s_N\big).

\subsection{Hypergeometric series}\label{s.hyper}
Fix $|q|<1$. Define the $q$-Pochhammer symbol
$$(a;q)_n := \prod_{i=0}^n (1-a q^i)\qquad  \textrm{and} \qquad (a;q)_{\infty} := \prod_{i=0}^{\infty} (1-a q^i).$$
We record three identities (which follow directly from the definition) satisfied by the $q$-Pochhammer symbol. These can be found as (1.2.30), (1.2.31) and (1.2.32) in \cite{Gaspar}.
\begin{align}\label{e.gasp}
\nonumber (A)\qquad& (a;q)_n = \frac{(a;q)_\infty}{(aq^n;q)_{\infty}},\\
(B)\qquad&(a^{-1} q^{1-n};q)_{n} = (a;q)_n (-a^{-1})^n q^{-\frac{n(n-1)}{2}},\\
\nonumber (C)\qquad& (a;q)_{n-k} = \frac{(a;q)_n}{(a^{-1} q^{1-n};q)_k} (-q a^{-1})^k q^{\frac{k(k-1)}{2} - nk}.
\end{align}

Define the basic $q$-hypergeometric series ${{_2\phi_1}}$ as
\begin{equation}\label{d.hypphi}
{_2\phi_1}(a,b;c;q,z) := \sum_{n=0}^{\infty} \frac{ (a;q)_{n} \,(b;q)_n}{(q;q)_n\, (c;q)_n} \,z^n.
\end{equation}
Since $|q|<1$, this is convergent for $|z|<1$. Heine's 1847 $q$-generalization of Gauss' summation formula \cite[Section 1.5]{Gaspar} states that
\begin{equation}\label{e.qgauss}
{_2\phi_1}(a,b;c;q,c/ab) = \frac{(c/a;q)_{\infty}\, (c/b;q)_{\infty}}{(c;q)_{\infty}\, (c/ab;q)_{\infty}},
\end{equation}
as long as $|c/ab|<1$. A special degeneration of this summation formula states that for any $n\geq 0$,
\begin{equation}\label{e.qgaussdeg}
{_2\phi_1}(q^{-n},b;c;q,q) = \frac{(c/b;q)_n}{(c;q)_n} b^n.
\end{equation}

\subsection{A $(q,\mu,\nu)$-deformed Binomial distribution}
We define the three parameter deformation of the Binomial distribution introduced in \cite{Pov} (see therein for references relating various applications for limits of this distribution). For $|q|<1$, $0\leq \nu\leq \mu<1$, and integers $0\leq j\leq m$ define the function
\begin{equation}\label{e.jumpdist}
\varphi_{q,\mu,\nu}(j|m) := \mu^j\, \frac{(\nu/\mu;q)_j (\mu;q)_{m-j}}{(\nu;q)_m}\,  \frac{(q;q)_m}{(q;q)_{j}(q;q)_{m-j}}.
\end{equation}
When $m=+\infty$, extend this definition by setting
\begin{equation}\label{e.jumpdistinfty}
\varphi_{q,\mu,\nu}(j|+\infty) := \mu^j\, \frac{(\nu/\mu;q)_j (\mu;q)_{\infty}}{(\nu;q)_{\infty}}\,  \frac{1}{(q;q)_{j}}.
\end{equation}

The following lemma shows that for each $m\in \Z_{\geq 0}\cup \{+\infty\}$ this defines a probability distribution on $j\in \{0,\ldots, m\}$.
\begin{lemma}\label{l.normalized}
Fix any choice of parameters $|q|<1$ and $0\leq \nu\leq \mu<1$. Then for $m\in \Z_{\geq 0}\cup \{+\infty\}$, the function $\varphi_{q,\mu,\nu}(j|m)$ defines a probability distribution over the set of $j\in \{0,\ldots, m\}$. Equivalently,
\begin{equation*}
\sum_{j=0}^m \varphi_{q,\mu,\nu}(j|m) = 1.
\end{equation*}
\end{lemma}
\begin{proof}
By using (\ref{e.gasp}(C)) the desired identity can be written as
$$
\sum_{j=0}^m \frac{(q^{-m};q)_j\, (\nu/\mu;q)_j}{(q;q)_j (\mu^{-1} q^{1-m};q)_j} q^j = \frac{(\nu;q)_m}{(\mu;q)_m}.
$$
The left-hand side is readily identified as ${_2\phi_1}(q^{-m},\nu/\mu; \mu^{-1} q^{1-m};q,q)$ (which terminates for $j>m$ due to the $q^{-m}$ argument). Applying (\ref{e.qgaussdeg}) with $b=\mu/\nu$ and $c=\mu^{-1} q^{1-m}$, and subsequently applying (\ref{e.gasp}(B)) yields the desired right-hand side above, and hence proves the identity.
\end{proof}

The following proposition generalizes Lemma \ref{l.normalized} (which corresponds to $y=0$) and is essentially paramount to the intertwining relationship we prove later as Theorem \ref{t.intertwine}. The proposition is proved in Section \ref{s.binidentityproof}. The $\nu=0$ version of this result was proved earlier as \cite[Lemma 3.7]{BCdiscrete}. The proof we present herein follows that general approach, though that presence of $\nu\neq 0$ necessitates our use of Heine's 1847 $q$-generalization of Gauss' summation formula, given above as (\ref{e.qgauss}).

\begin{proposition}\label{p.dualitymatch}
Fix any choice of parameters $|q|<1$ and $0\leq \nu\leq \mu<1$. Then. for all $m,y\in \Z_{\geq 0}$,
$$
\sum_{j=0}^{m} \varphi_{q,\mu,\nu}(j|m) q^{jy} = \sum_{s=0}^{y} \varphi_{q,\mu,\nu}(s|y) q^{sm}.
$$
Similarly, for all $y\in \Z_{\geq 0}$,
$$
\sum_{j=0}^{+\infty} \varphi_{q,\mu,\nu}(j|+\infty) q^{jy} = \varphi_{q,\mu,\nu}(0|y).
$$
\end{proposition}

Both Lemma \ref{l.normalized} and Proposition \ref{p.dualitymatch} can be analytically continued in the parameter $q$ within a suitable domain, though we do not expand on this observation or its applications herein.

\subsection{The $(q,\mu,\nu)$-Boson process}\label{s.qboson}
This discrete time interacting particle system was introduced by Povolotsky \cite{Pov} and shown therein to be the most general zero range chipping model \cite{Evans} with factorized steady state which is also solvable via coordinate Bethe ansatz. See Sections \ref{s.stat}, \ref{s.plan} and \ref{s.limitproc} for further discussion on this; Section \ref{s.genpara} for a general parameter version of this process for which some of our results still hold; and Section \ref{s.aba} for a discussion related to our choice of using the term {\it Boson} in naming this process.

Fix $|q|<1$, $0\leq \nu\leq \mu<1$ and an integer $N\geq 1$. The $N$-site $(q,\mu,\nu)$-Boson process is a discrete time Markov chain $\vec{y}(t) = \{y_i(t)\}_{i=0}^N\in \YY^N$. The values of $y_i(t)$ record the number of particles above site $i$ at time $t$. At time $t+1$ the state $\vec{y}(t)$ is updated to another state $\vec{y}(t+1)$ according to the following procedure. For each site $i\in \{1,\ldots,N\}$,  $s_i\in \big\{0,1,\ldots, y_i(t)\big\}$ particles are transferred to site $i-1$ (at time $t+1$) with probability $\varphi_{q,\mu,\nu}(s_i|y_i(t))$ (see the left-hand side of Figure \ref{f.particlefigures}). This occurs in parallel for each site, so $y_i(t+1)$ is equal to $y_i(t)$ plus those particles which were transferred from site $i+1$ to site $i$ and minus those particles which were transferred from site $i$ to site $i-1$ (i.e. $y_i(t+1) = y_i(t) + s_{i+1}-s_i$). No particles are transferred into site $N$ or out of site $0$.

\begin{figure}
\centering\epsfig{file=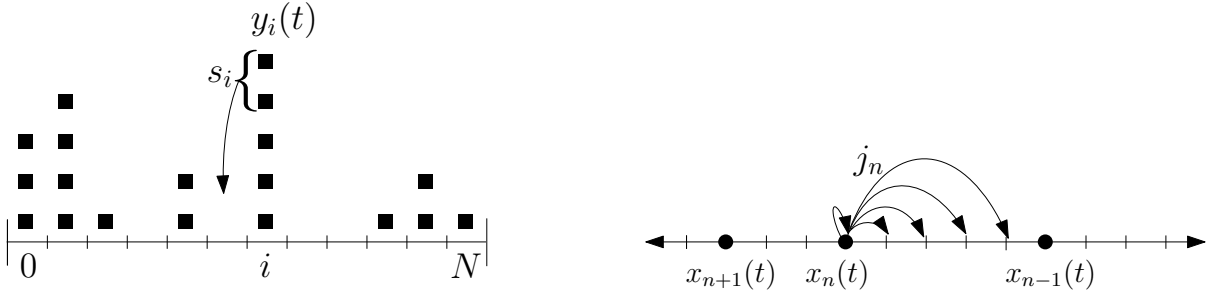, width=16cm}
\caption{Left: A single time step of the $(q,\mu,\nu)$-Boson process is illustrated. For each site $i\in\{1,\ldots, N\}$, a random number of the $y_i(t)$ particles above the site are moved them to the left to site $i-1$. The jumps are independent at each site and distributed according to $\varphi_{q,\mu,\nu}\big(j|y_i(t)\big)$. Right: A single time step of the $(q,\mu,\nu)$-TASEP is illustrated. Each particle $x_i(t)$ with $i\in\{1,\ldots, N\}$ jumps to any of the sites $x_n(t)$ through $x_{n+1}(t)-1$. The jumps are independent at each site and their length $j$ is distributed according to $\varphi_{q,\mu,\nu}\big(j|x_{n-1}(t)-x_n(t)-1\big)$}\label{f.particlefigures}
\end{figure}

For functions $f:\YY^N\to \R$ define the operators $\big[\mathcal{A}_{q,\mu,\nu}\big]_i$, $i\in \{1,\ldots,N\}$, via their action
\begin{equation*}
\big[\mathcal{A}_{q,\mu,\nu}\big]_i f(\vec{y}) = \sum_{s_i=0}^{y_i} \varphi_{q,\mu,\nu}(s_i|y_i) f\big(\vec{y}^{s_i}_{i,i-1}\big).
\end{equation*}
The operator $\big[\mathcal{A}_{q,\mu,\nu}\big]_i$ encodes the movement of particles from site $i$ to $i-1$. The single time step transition matrix $P^{{\rm Boson}}$ for the $N$-site $(q,\mu,\nu)$-Boson process (which is time homogeneous) is written via the product of these operators
$$
\big(P^{{\rm Boson}} f\big)(\vec{y}) =\big[\mathcal{A}_{q,\mu,\nu}\big]_1\cdots \big[\mathcal{A}_{q,\mu,\nu}\big]_N f(\vec{y}).
$$

It is clear that the $N$-site $(q,\mu,\nu)$-Boson process conserves the number of particles. Hence the $N$-site, $k$-particle $(q,\mu,\nu)$-Boson process has the state space $\YY^N_k$. Rather than identifying the number of particles per site, we may identify the state $\vec{y}$ via recording the ordered location $\vec{n}(\vec{y})\in \Weyl^N_k$ of each of the $k$ particles.

\subsection{The $(q,\mu,\nu)$-TASEP}\label{s.qtasep}
This discrete time interacting particle system was introduced by Povolotsky \cite{Pov} by virtue of its relationship to the $(q,\mu,\nu)$-Boson process via a simple ZRP-ASEP and particle-hole transformation (see Remark \ref{r.otherdual} herein).

Fix $|q|<1$, $0\leq \nu\leq \mu<1$ and an integer $N\geq 1$. The $N$-particle $(q,\mu,\nu)$-TASEP (totally asymmetric simple exclusion process) is a discrete time Markov chain $\vec{x}(t) =\{x_n(t)\}_{n=0}^{N}\in \XX^N$. The purpose of setting $x_0\equiv +\infty$ is to simplify notation via having a ``virtual particle'' at $+\infty$. The value of $x_n(t)$ records the location of particle $n$ at time $t$. At time $t+1$ the state $\vec{x}(t)$ is updated to another state $\vec{x}(t+1)$ according to the following procedure. For each $n\in\{1,\ldots, N\}$, $x_{n}(t)$ updates in parallel and independently to $x_n(t+1) = x_n(t) + j_n$ where $j_n\in \big\{0,\ldots, x_{n-1}(t)-x_{n}(t)-1 \big\}$ is drawn according to the probability distribution $\varphi_{q,\mu,\nu}(j_n|x_{n-1}(t)-x_{n}(t)-1)$ (see the right-hand side of Figure \ref{f.particlefigures}).

For functions $f:\XX^N\to \R$ define the operators $\big[\mathcal{B}_{q,\mu,\nu}\big]_n$,  $n\in \{1,\ldots,N\}$, via their action
\begin{equation*}
\big[\mathcal{B}_{q,\mu,\nu}\big]_n f(\vec{x}) = \sum_{j_n=0}^{x_{n-1}-x_n-1} \varphi_{q,\mu,\nu}(j_n|x_{n-1}-x_n-1) f\big(\vec{y}^{j_n}_{n}\big).
\end{equation*}
This operator encodes the one-step update of the location of particle $x_n(t)$. The single time step transition matrix $P^{{\rm TASEP}}$ for $N$-particle $(q,\mu,\nu)$-TASEP (which is time homogeneous) is written via the product of these operators
$$
\big(P^{{\rm TASEP}} f\big)(\vec{x}) =\big[\mathcal{B}_{q,\mu,\nu}\big]_1\cdots \big[\mathcal{B}_{q,\mu,\nu}\big]_N f(\vec{x}).
$$

The update rule for $(q,\mu,\nu)$-TASEP is one-sided in the sense that particle $n$ only depends on particle $n-1$ (and hence overtime all particles with lower index). Therefore, we may consider half-infinite $(q,\mu,\nu)$-TASEP in which configurations are given by particle locations $x_1(t)>x_2(t)>\cdots$. Any event concerning particle $x_N(t)$ depends only upon the evolution of the first $N$ particles and hence can be studied in reference to the $N$-particle process. We will be concerned herein with {\it step initial data} where $x_n(0)=-n$ for $n\geq 1$. See Sections \ref{s.stat} and \ref{s.plan} for mention of some additional types of initial data.

\subsection{Intertwining the $(q,\mu,\nu)$-Boson process and $(q,\mu,\nu)$-TASEP}
We come now to the primary contribution of this paper. Theorem \ref{t.intertwine} and its corollaries demonstrate an {\it intertwining} relationship between the $(q,\mu,\nu)$-Boson process and $(q,\mu,\nu)$-TASEP, generalizing the dualities discovered by \cite{BCS,BCdiscrete} between the continuous Poison and discrete geometric $q$-TASEP and $q$-Boson processes. The following theorem is a special case ($a_i\equiv 1$ and $\mu_t\equiv \mu$) of Theorem \ref{t.genintertwine}. Since we prove that general statement later, we forego a proof of this result here. We do, however, remark that this essentially boils down to Proposition \ref{p.dualitymatch}.

\begin{theorem}\label{t.intertwine}
Fix $|q|<1$, $0\leq \nu\leq \mu<1$ and an integer $N\geq 1$. Define $H:\XX^N\times \YY^N\to \R$ as
\begin{equation}\label{e.hfunct}
H(x;y) := \prod_{i=0}^{N} q^{y_i(x_i+i)},
\end{equation}
with the convention that if $y_0>0$ then the above product is 0. Then $H$ {\it intertwines} the $N$-site $(q,\mu,\nu)$-Boson process and $N$-particle $(q,\mu,\nu)$-TASEP in the sense that
$$
P^{{\rm TASEP}} H  = H \big(P^{{\rm Boson}}\big)^{\top},
$$
where $\big(P^{{\rm Boson}}\big)^{\top}$ represents the transpose of $P^{{\rm Boson}}$.
\end{theorem}

We record two corollaries of the above intertwining. In order to state the first corollary, we fix the following notation. We say $h:\Z_{\geq 0}\times \YY^N\to \R_{\geq 0}$ solves the {\it true evolution equation} with initial data $h_0(\vec{y})$ if it satisfies:
\begin{enumerate}
\item for all $\vec{y}\in \YY^N$ and $t\geq 0$, $h(t+1;\vec{y}) = P^{{\rm Boson}}h(t;\vec{y})$;
%\item For all $\vec{y}\in \YY^N$ such that $y_0>0$, $h(t;\vec{y}) \equiv 0$ for all $t\geq 0$;
\item for all $\vec{y}\in \YY^N$, $h(0;\vec{y}) = h_0(\vec{y})$.
\end{enumerate}

The upcoming corollary follows quite readily from Theorem \ref{t.intertwine} -- see the proof of Corollary \ref{c.truegen}, a general parameter version of this result. It shows that a certain family of expectations of observables of the $(q,\mu,\nu)$-TASEP satisfy a closed, deterministic evolution equation.

\begin{corollary}\label{c.trueevo}
For any fixed $\vec{x}\in \XX^N$,
$$h(t;\vec{y}):= \EE^{\vec{x}}\big[H(\vec{x}(t),\vec{y}) \big]=\EE^{\vec{x}}\Big[\prod_{i=0}^{N} q^{y_i(x_i+i)}\Big]$$
is the unique solution to the true evolution equation with initial data $h_0(\vec{y}):= \prod_{i=0}^{N} q^{y_i(x_i+i)}$.
\end{corollary}

The second corollary (which we state but do not utilize) is the Markov duality of the $(q,\mu,\nu)$-Boson process and $(q,\mu,\nu)$-TASEP. Recall that in general, two Markov chains $x(t)$ and $y(t)$ with state spaces $X$ and $Y$ are said to be {\it dual} with respect to a {\it duality functional} $H:X\times Y\to \R$ if for all $x\in X$ and $y\in Y$, and all $t\geq 0$
$$
\EE^{x}\Big[H\big(x(t),y\big)\Big] = \EE^{y}\Big[H\big(x,y(t)\big)\Big],
$$
where $\EE^{x}$ is the expectation of the Markov chain $x(t)$ stated from $x(0)=x$, and $\EE^{y}$ is the expectation of the Markov chain $y(t)$ stated from $y(0)=y$.
\begin{corollary}\label{c.dual}
The $N$-particle $(q,\mu,\nu)$-TASEP $\vec{x}(t)$ and $N$-site $(q,\mu,\nu)$-Boson process $\vec{y}(t)$ are dual with respect to the duality functional $H(x;y)$.
\end{corollary}

\begin{remark}\label{r.otherdual}
There is a simpler duality between these processes which should be distinguished from that of Corollary \ref{c.dual}. The gaps $g_i(t) := x_{i-1}(t)-x_{i}(t)$ of the $(q,\mu,\nu)$-TASEP evolve according to the update rule for the $(q,\mu,\nu)$-Boson process in which particles now jump from site $i$ to $i+1$ (instead of $i$ to $i-1$). For instance, step initial data $(q,\mu,\nu)$-TASEP corresponds with $g_1(0)=+\infty$ and $g_i(0) \equiv 0$ for $i\neq 1$. This mapping between the two processes was discussed by Povolotsky \cite{Pov} and called a ZRP-ASEP and particle-hole transformation.
\end{remark}

\subsection{Free evolution equations with two-body boundary conditions}\label{s.betheant}
We confront the question of how to solve the true evolution equation in Corollary \ref{c.trueevo} and hence compute formulas for the expectations of $(q,\mu,\nu)$-TASEP observables. A priori it is not clear how to proceed. The first reduction is to recognize that the transition matrix $P^{{\rm Boson}}$ is a direct sum of transition matrices for the $N$-site, $k$-particle $(q,\mu,\nu)$-Boson process which has state space $\YY^N_k$. In principal this reduces the problem of computing $h(t;\vec{y})$ for any fixed $N$ and $k$ to a matter of finite matrix multiplication (or exponentiation). The challenge, however, is to figure out a way to perform this computation with complexity which remains constant as $N$, $k$ and $t$ grow.

Proposition \ref{p.freeevo} provides an important step towards this reduction in complexity. The $N$-site, $k$-particle $(q,\mu,\nu)$-Boson process $\vec{y}(t)$ can be rewritten in terms of particle locations $\vec{n}(t)=\vec{n}(\vec{y}(t))$. The idea (which dates back to Bethe's solution \cite{Bethe} of the Heisenberg XXX quantum spin chain) is to rewrite the $k$-particle true evolution equation for $\vec{n}(t)\in \Weyl_k^N$ in terms of a $k$-particle {\it free evolution equation} subject to $k-1$ {\it two-body boundary conditions}, but on a large state space $\vec{n}\in \Z^k$. In the spirit of the reflection principal, any solution to the free evolution equation, which satisfies the two-body boundary conditions and satisfies the desired initial data when restricted to $\Weyl_k^N$ will then coincide on $\Weyl_k^N$ with the unique solution to the true evolution equation.

Most every $k$-particle system do not enjoy this reducibility as higher order boundary conditions must also be imposed. The reason Povolotsky introduced the $(q,\mu,\nu)$-Boson process in \cite{Pov} was because it was the most general process within the class he was considering which enjoyed this property. The following proposition is effectively contained in Section 3 of \cite{Pov}. It is a special case of Proposition \ref{p.freeevogen}, which contains details as to the proof.

\begin{proposition}\label{p.freeevo}
Fix $|q|<1$, $0\leq \nu\leq \mu<1$, and integers $N,k\geq 1$. If $u:\R_{\geq 0} \times \Z^k \to \C$ solves:
\begin{enumerate}
\item {\em ($k$-particle free evolution equation)} for all $\vec{n}\in \Z^k$ and $t\geq 0$,
$$
u(t+1;\vec{n}) = \prod_{i=1}^{k} \big[\nabla_{\mu,\nu}\big]_i u(t;\vec{n}),
$$
where $[\nabla_{\mu,\nu}]_i u(t;\vec{n}) := \frac{\mu-\nu}{1-\nu} u(t;\vec{n}_{i}^{-}) + \frac{1-\mu}{1-\nu} u(t;\vec{n})$;
\item {\em ($k-1$ two-body boundary conditions)} for all $\vec{n}\in \Z^k$ such that for some $i\in \{1,\ldots, k-1\}$, $n_i= n_{i+1}$, and all $t\geq 0$,
$$
\alpha u(t; n_{i,i+1}^{-}) +\beta u(t;\vec{n}_{i+1}^{-}) + \gamma u(t;\vec{n}) -u(t;\vec{n}_i^{-}) = 0
$$
where the parameters $\alpha,\beta,\gamma$ are defined in terms of $q$ and $\nu$ as
$$
\alpha = \frac{\nu(1-q)}{1-q\nu},\qquad \beta= \frac{q-\nu}{1-q\nu},\qquad \gamma = \frac{1-q}{1-q\nu};
$$
\item {\em (initial data)} for all $\vec{n}\in \Weyl^N_k$, $u(0;\vec{n}) = h_0\big(\vec{y}(\vec{n})\big)$;
\end{enumerate}
then for all $\vec{n}\in \Weyl_k^N$, and all $t\geq 0$,
$
u(t;\vec{n}) = h\big(t;\vec{y}(\vec{n})\big)
$
where $h(t;\vec{y})$ is the solution to the true evolution equation with initial data $h_0$.
\end{proposition}

\subsection{Nested contour integral formula for step initial data}
Theorem \ref{t.moments} provides an exact, and concise nested contour integral formula for the expectations of the general class of observables considered in Corollary \ref{c.trueevo}, for the $(q,\mu,\nu)$-TASEP started from step initial data. This achieves the aim of finding a solution to the true evolution equation which does not grow in complexity. At first glance, this formula might seem to be pulled out of thin air. Formulas of this type, however, have occurred previously in the coordinate Bethe ansatz literature and the Macdonald process literature -- see Sections \ref{s.plan} and \ref{s.mac} for further discussion as well as \cite{Yudson, BorCor,BCS,BCdiscrete, BCPS}. In any case, once given such a proposed formula, assisted by Proposition \ref{p.freeevo} it is quite simple to prove the theorem.

\begin{theorem}\label{t.moments}
Fix $|q|<1$, $0\leq \nu\leq \mu<1$, and integers $N,k\geq 1$. Consider $(q,\mu,\nu)$-TASEP started from step initial data. Then for any $\vec{n}\in \Weyl_k^N$,
\begin{equation}\label{lhsthm}
\EE\left[\prod_{i=1}^{k} q^{x_{n_i}+n_i}\right] = \frac{(-1)^k q^{\frac{k(k-1)}{2}}}{(2\pi \i)^k} \oint_{\gamma_1} \cdots \oint_{\gamma_k} \prod_{1\leq A<B\leq k} \frac{z_A-z_B}{z_A-q z_B} \, \prod_{j=1}^{k} \left(\frac{1-\nu z_j}{1-z_j}\right)^{n_j} \left(\frac{1-\mu z_j}{1-\nu z_j}\right)^t \frac{dz_j}{z_j(1-\nu z_j)}
\end{equation}
where the integration contours $\gamma_1,\ldots, \gamma_k$ are chosen so they all contain $1$, $\gamma_A$ contains $q\gamma_B$ for $B>A$ and all contours exclude $0$ and $1/\nu$ (see Figure \ref{f.contours} for an illustration when $k=5$).
\end{theorem}

\begin{figure}
\centering\epsfig{file=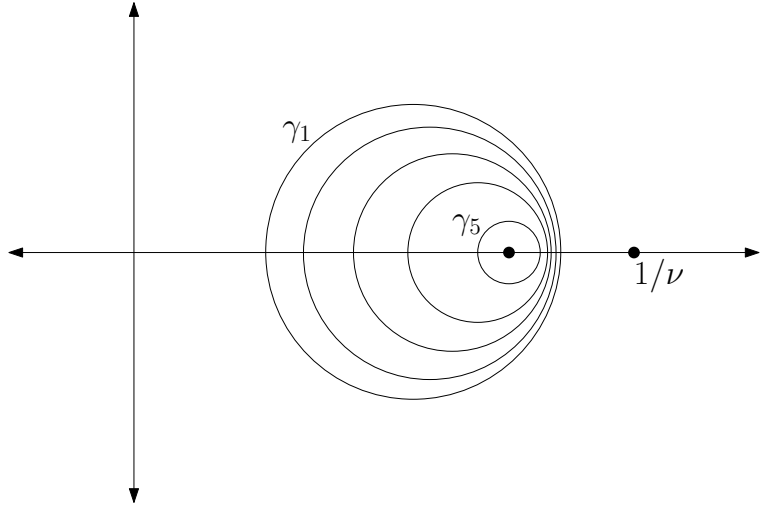, width=10cm}
\caption{Possible contours of integration for Theorem \ref{t.moments} with $k=5$}\label{f.contours}
\end{figure}

\begin{proof}
Call the right-hand side of (\ref{lhsthm}) $u(t;\vec{n})$. Let us check that $u(t;\vec{n})$ satisfies the three conditions of Proposition \ref{p.freeevo} with initial data $h_0(\vec{y}) = \mathbf{1}_{\vec{y}:y_0=0}$ (or in other words $u(0;\vec{n}) = \prod_{i=1}^{k} \mathbf{1}_{n_i>0}$).

To check the $k$-particle free evolution equation observe that by bring $\big[\nabla_{\mu,\nu}\big]_i$ inside the integration, it suffices to confirm that
$$
\big[\nabla_{\mu,\nu}\big]_i \left(\frac{1-\nu z_i}{1-z_i}\right)^{n_i} \left(\frac{1-\mu z_i}{1-\nu z_i}\right)^t = \left(\frac{1-\nu z_i}{1-z_i}\right)^{n_i} \left(\frac{1-\mu z_i}{1-\nu z_i}\right)^{t+1},
$$
as is readily done.

To check the $k-1$ two-body boundary conditions we may apply the boundary condition to the integrand. This application brings out an additional factor in the integrand of
$$
\frac{(1-\nu)^2}{(1-q\nu) (1-\nu z_{i})(1-\nu z_{i+1})} \, (z_{i}-q z_{i+1}).
$$
The factor of $(z_{i}-q z_{i+1})$ cancels the pole separating the contours $\gamma_i$ and $\gamma_{i+1}$. We are now free to deform the contours $\gamma_i$ and $\gamma_{i+1}$ to lie along the same curve. Since $n_i=n_{i+1}$ (by hypothesis) the integrand is anti-symmetric in $z_i$ and $z_{i+1}$. This, however, implies that the integral is zero, and hence the second condition is confirmed.

To check the initial data observe that when $n_k= 0$, there is no pole at $1$ for the $z_k$ integral. Thus, since $\gamma_k$ contains no poles, the integral (and likewise $u(0;\vec{n})$) is zero. The alternative $n_k>0$ implies all $n_i>0$ for $\vec{n}\in \Weyl_k^N$. In this case, we expand $\gamma_1$ through $\gamma_k$ to infinity. There are no poles at $1/\nu$ since $n_i-1\geq 0$, and there is no pole at infinity due to quadratic decay. There are poles at $z_i=0$ which (through evaluating the residues) shows the integral is equal to one. Thus we have shown  $u(0;\vec{n}) = \prod_{i=1}^{k} \mathbf{1}_{n_i>0}$ as desired.

By virtue of Proposition \ref{p.freeevo}, this means that $u(t;\vec{n}) = h\big(t;\vec{y}(\vec{n})\big)$ where $h(t;\vec{y})$ is the solution to the true evolution equation with initial data $h_0(\vec{y}) = \mathbf{1}_{\vec{y}:y_0=0}$. Corollary \ref{c.trueevo} then implies that $u(t;\vec{n}) =  \EE^{\vec{x}}\big[H(\vec{x}(t),\vec{y}(\vec{n})) \big]$ which is immediately matched to the left-hand side of (\ref{lhsthm}).
\end{proof}

\subsection{Fredholm determinant formula}

For $(q,\mu,\nu)$-TASEP with step initial data, the formulas from Theorem \ref{t.moments} provide a complete characterization of the distribution of $\vec{x}(t)$. Indeed, each random variable $q^{x_{n}(t) +n}$, $1\leq n\leq N$, is in $(0,1)$ and hence knowledge of all joint movement suffices to characterize the joint distribution. Despite this fact, it is not obvious how to extract meaningful asymptotic distribution information from these formulas. In the case of one-point distributions (i.e. the distribution of $x_n(t)$ for a single $n$) this was achieved in \cite{BorCor}. We will apply the approach developed in \cite{BorCor} (in particular, the general restatement of the calculation in \cite{BorCor} which can be found in \cite[Section 3]{BCS}).

Theorem \ref{t.fred} provides two Fredholm determinant formulas for the $e_q$-Laplace transform of the observable $q^{x_{n}(t)+n}$, and consequently for the one-point distribution of $x_n(t)$ (see Remark \ref{r.invert}).  This type of Fredholm determinant formula (in particular that of (\ref{MellinBarnes}) is quite amenable to asymptotic analysis -- see Section \ref{s.limitproc} for further discussion.

\begin{theorem}\label{t.fred}
Fix $q\in (0,1)$ and $0\leq \nu\leq \mu<1$. Consider $(q,\mu,\nu)$-TASEP $\vec{x}(t)$ started from step initial data. Then for all $\zeta\in\C\setminus \R_+$,
\begin{equation}\label{MellinBarnes}
\EE \left[\frac{1}{\big(\zeta q^{x_{n}(t)+n};q\big)_{\infty}}\right] = \det\big(I + K_{\zeta}\big)
\end{equation}
where $\det\big(I + K_{\zeta}\big)$ is the Fredholm determinant of $K_\zeta: L^2(C_1)\to L^2(C_1)$ for $C_1$ a positively oriented circle containing 1 with small enough radius so as to not contain 0, $1/q$ and $1/\nu$. The operator $K_\zeta$ is defined in terms of its integral kernel
\begin{equation*}
K_{\zeta}(w,w') = \frac{1}{2\pi \i} \int_{-\i \infty + 1/2}^{\i\infty +1/2} \frac{\pi}{\sin(-\pi s)} (-\zeta)^s \frac{g(w)}{g(q^s w)} \frac{1}{q^s w - w'} ds
\end{equation*}
with
\begin{equation*}
g(w) = \left(\frac{(\nu w;q)_{\infty}}{(w;q)_{\infty}}\right)^{n} \left( \frac{(\mu w;q)_{\infty}}{(\nu w;q)_{\infty}}\right)^t \frac{1}{(\nu w;q)_{\infty}}.
\end{equation*}

The following second formula also holds:
\begin{equation}\label{Cauchy}
\EE \left[\frac{1}{\big(\zeta q^{x_{n}(t)+n};q\big)_{\infty}}\right] = \frac{\det\big(I + \zeta \tilde{K}\big)}{(\zeta;q)_{\infty}}
\end{equation}
where $\det\big(I + \zeta \tilde{K}\big)$ is the Fredholm determinant of $\zeta$ times the operator $\tilde{K}_\zeta: L^2(C_{0,1})\to L^2(C_{0,1})$ for $C_{0,1}$ a positively oriented circle containing 0 and 1 (but not $1/\nu$). The operator $\tilde{K}$ is defined in terms of its integral kernel
\begin{equation*}
\tilde{K}(w,w') = \frac{g(w)/g(qw)}{qw'-w}
\end{equation*}
where the function $g$ is as above.
\end{theorem}
As this type of deduction of Fredholm determinant formulas from $q$-moment formulas has appeared before in \cite{BorCor,BCS,BCdiscrete}, we provide the steps of the proof, without going into too much detail as to how they are justified. We also do not recall the definition of Fredholm determiants, but rather refer the reader to \cite[Section 3.2.2]{BorCor}.

\begin{proof}
The first Fredholm determinant formula is referred to by \cite{BCS} as {\it Mellin-Barnes} type, and the second as {\it Cauchy} type. In order to prove the Mellin-Barnes type formula we utilize the formula for $\EE\big[q^{k(x_n(t)+n)}\big]$ from specializing all $n_i\equiv n$ in Theorem \ref{t.moments}. For the purpose of this proof define $\mu_k := \EE\big[q^{k(x_n(t)+n)}\big]$. The reason we permit this abuse of notation (this $\mu_k$ of course has a different meaning than the parameter $\mu$, or $\mu_t$ from Section \ref{s.discuss}) is that we can now identify $\mu_k$ with the formula present in \cite[Definition 3.1]{BCS}, subject to defining $f(w):= g(w)/g(qw)$, with $g$ from the statement of Theorem \ref{t.fred}. We may then apply \cite[Propositions 3.3 and 3.6]{BCS} with the contour $C_{A}=C_1$ chosen to be a very small circle around 1, and $D_{R,d}, D_{R,d;k}$ specified by setting $R=1/2$ (and $d$ arbitrary, as it does not matter for this choice of $\R$). The outcome of this is that
$$
\sum_{k\geq 0} \mu_k \frac{\zeta^k}{k_q!} =  \det\big(I + K_{\zeta}\big).
$$
In the course of this application, it is necessary to check that a few technical conditions on the contours, as well as $\zeta$ and $g$ are satisfied. These are easily confirmed for $\zeta$ with $|\zeta|$ small enough, and $C_{1}$ a small enough circle around 1. The only condition depending on the function $g$ is that $|g(w)/g(q^sw)|$ remain uniformly bounded as $w\in C_{1}$, $k\in\Z_{>0}$ and $s\in D_{R,d;k}$ varies. This is readily confirmed for $g$ from the statement of Theorem \ref{t.fred}.

Now, observe that for $\zeta$ with $|\zeta|$ small enough, we also have that
$$
\sum_{k\geq 0} \mu_k \frac{\zeta^k}{k_q!}= \EE \left[\frac{1}{\big(\zeta q^{x_{n}(t)+n};q\big)_{\infty}}\right].
$$
This is justified (as in \cite[Theorem 3.2.11]{BorCor}) by the deterministic fact that $q^{x_{n}(t)+n}\in (0,1)$ and an application of the $q$-Binomial theorem. This establishes (\ref{MellinBarnes}) for $|\zeta|$ sufficiently small. However, both sides are analytic over $C\setminus \R_{+}$ and thus the claimed Mellin-Barnes type result of (\ref{MellinBarnes}) follows via analytic continuation.

The Cauchy type formula (\ref{Cauchy}) follows from \cite[Proposition 3.10]{BCS} along with a small amount of algebra. The proof essentially follows that of \cite[Theorem 3.2.16]{BorCor}.
\end{proof}

\begin{remark}\label{r.invert}
Just like the usual Laplace transform of a positive random variable, the $e_q$-Laplace transform in Theorem \ref{t.fred} can be readily inverted (see \cite[Proposition 3.1.1]{BorCor} or \cite[Proposition 7.1]{BCS}) to give the distribution of $x_n(t)$.
\end{remark}
\begin{remark}
Setting $g_i(t) := x_{i-1}(t)-x_{i}(t)$, Remark \ref{r.otherdual} shows that $\vec{g}(t)$ evolves as the $(q,\mu,\nu)$-Boson process with particles moving from $i$ to $i+1$. The $(q,\mu,\nu)$-TASEP step initial data corresponds with having $g_1(0)=+\infty$ and $g_i(0)=0$ for $i>1$. Let $C_s(t) = \sum_{i=s+1}^{\infty} g_i(t)$ be the number of particles of $\vec{g}(t)$ strictly to the right of site $s$ at time $t$. Then clearly $\{x_n(t)+n\geq s\} = \{C_{s}(t)\geq n\}$ and hence Theorem \ref{t.fred}, in light of the inversion formula mentioned in Remark \ref{r.invert}, provides an exact formula for the distribution of $C_s(t)$ as well.
\end{remark}

\subsection{Acknowledgements}
The author extends thanks for A. Povolotsky for an early draft of the paper \cite{Pov} as well as useful discussions, and also appreciates discussions with A. Borodin and B. Vet\H{o} related to this work. The author was partially supported by the NSF through DMS-1208998 as well as by Microsoft Research and MIT through the Schramm Memorial Fellowship, and by the Clay Mathematics Institute through the Clay Research Fellowship.

\section{General parameter intertwining}\label{s.genparsection}

We introduce a more general version of the $(q,\mu,\nu)$-Boson process and $(q,\mu,\nu)$-TASEP which includes site/particle dependent jump parameters $a_i$, $i\in \{1,\ldots, N\}$ and time dependent jump parameters $\mu_t$, $t\in \Z_{>0}$. The processes considered in Section \ref{s.intro} correspond to setting all $a_i\equiv 1$ and all $\mu_t\equiv \mu$. We state and prove a general parameter analog of the intertwining relationship given earlier as Theorem \ref{t.intertwine} (below Theorem \ref{t.genintertwine}) and an analog of Proposition \ref{p.freeevo} (below Proposition \ref{p.freeevogen}) providing the reduction of the associate true evolution equation to a free evolution equation with two-body boundary conditions. For general $\mu_t$ but $a_i\equiv 1$ we provide the modifications to Theorems \ref{t.moments} and \ref{t.fred}. Presently it is not clear whether analogous results hold in the full generality of varying $a_i$ parameters. For the case $\nu=0$ analogs of Theorem \ref{t.moments} and \ref{t.fred} do hold \cite{BCdiscrete} for general $a_i$ parameters -- see Section \ref{s.mac} for a brief discussion of the relation between the inclusion of these parameters and Macdonald processes \cite{BorCor}.

\subsection{Site/particle dependent and time dependent jump parameters}\label{s.genpara}
Fix $|q|<1$, $\nu\in [0,1)$, and an integer $N\geq 1$. Further, fix site/particle dependent jump parameters $a_i>0$ for all $i\in \{1,\ldots, N\}$, and time dependent jump parameters $\mu_t\in[\nu,1)$ for all $t\in \Z_{>0}$. We define a general parameter version of the $(q,\mu,\nu)$-Boson process and $(q,\mu,\nu)$-TASEP from Sections \ref{s.qboson} and \ref{s.qtasep} with respect to these site/particle dependent and time dependent jump parameters as follows. For the general parameter version of the $(q,\mu,\nu)$-Boson process, $s_i\in \{0,1,\ldots, y_i(t)\}$ particles are transferred from site $i$ to site $i-1$ at time $t+1$ with probability $\varphi_{q,a_i \mu_{t+1},\nu}(s_i|y_i(t))$. For the general parameter version of the $(q,\mu,\nu)$-TASEP, the particle $x_{n}(t)$ updates its location to $x_n(t+1) = x_n(t) + j_n$ where $j_n\in \big\{0,\ldots, x_{n-1}(t)-x_{n}(t)-1\big\}$ is drawn according to the probability distribution $\varphi_{q,a_n\mu_{t+1},\nu}(j_n|x_{n-1}(t)-x_{n}(t)-1)$.

These general parameter Markov chains are no-longer time homogeneous nor spatially homogeneous. For $t\geq 1$ let $P_t^{{\rm Boson}}$ and $P_t^{{\rm TASEP}}$ denote the respective transition matrices from time $t-1$ to time $t$ for the general parameter $(q,\mu,\nu)$-Boson process and $(q,\mu,\nu)$-TASEP. Generalizing what is written in Sections \ref{s.qboson} and \ref{s.qtasep}, respectively, we may write explicitly
\begin{eqnarray}\label{e.pts}
\big(P_t^{{\rm Boson}} f\big)(\vec{y}) &=&\big[\mathcal{A}_{q,a_1\mu_t,\nu}\big]_1\cdots \big[\mathcal{A}_{q,a_N\mu_t,\nu}\big]_N f(\vec{y}),\\
\nonumber\big(P^{{\rm TASEP}} f\big)(\vec{x}) &=&\big[\mathcal{B}_{q,a_1\mu_t,\nu}\big]_1\cdots \big[\mathcal{B}_{q,a_N\mu_t,\nu}\big]_N f(\vec{x}),\\
\end{eqnarray}
where in the first line $f:\YY^N\to \R$ and in the second line $f:\XX^N\to \R$.

\subsection{Intertwining}
The following theorem is a general parameter analog of Theorem \ref{t.intertwine}. The proof essentially amounts to Proposition \ref{p.dualitymatch}.

\begin{theorem}\label{t.genintertwine}
Fix $|q|<1$, $\nu\in [0,1)$, and an integer $N\geq 1$. Further, fix site/particle dependent jump parameters $a_i>0$ for all $i\in \{1,\ldots, N\}$, and time dependent jump parameters $\mu_t\in[\nu,1)$ for all $t\in \Z_{>0}$. Recall $H:\XX^N\times \YY^N\to \R$ from (\ref{e.hfunct}). $H$ {\it intertwines} the general parameter version of the $N$-particle $(q,\mu,\nu)$-TASEP $\vec{x}(t)$ and the the $N$-site $(q,\mu,\nu)$-Boson process $\vec{y}(t)$ in the sense that for all $t\geq 1$
$$
P_t^{{\rm TASEP}} H  = H \big(P_t^{{\rm Boson}}\big)^{\top}.
$$
\end{theorem}
\begin{proof}
Recalling the definition of $\big[\mathcal{B}_{q,\mu,\nu}\big]_i$ from Section \ref{s.qtasep} and employing (\ref{e.pts}) we readily find that
$$
P_t^{{\rm TASEP}} H(\vec{x},\vec{y}) = \prod_{i=1}^{N} \Bigg(\sum_{j_i=0}^{x_{i-1}-x_i-1} \varphi_{q,a_i\mu_t,\nu}(j_i|x_{i-1}-x_i-1) \, q^{j_i y_i}\Bigg) \, \prod_{i=0}^{N} q^{y_i(x_i+i)}.
$$
Applying Proposition \ref{p.dualitymatch} to each term of the product $i=1$ through $N$ above we find that
\begin{eqnarray*}
P_t^{{\rm TASEP}} H(\vec{x},\vec{y}) &=& \varphi_{q,a_1\mu_t,\nu}(0|y_1) \prod_{i=2}^{N} \Bigg( \sum_{s_i}^{y_i} \varphi_{q,a_i\mu_t,\nu}(s_i|y_i) q^{s_i (x_{i-1}-x_i-1)}\Bigg) \prod_{i=0}^{N} q^{y_i(x_i+i)}\\
&=& H \big(P_t^{{\rm Boson}}\big)^{\top},
\end{eqnarray*}
where we have employed (\ref{e.pts}) and the definition of $[\mathcal{A}_{q,\mu,\nu}\big]_i$ from Section \ref{s.qboson}.
\end{proof}

We say that $h:\Z_{\geq 0}\times \YY^N\to \R_{\geq 0}$ solves the general parameter version of the {\it true evolution equation} with initial data $h_0(\vec{y})$ if:
\begin{enumerate}
\item for all $\vec{y}\in \YY^N$ and $t\geq 0$, $h(t+1;\vec{y}) = P_{t+1}^{{\rm Boson}} h(t;\vec{y})$;
%\item For all $\vec{y}\in \YY^N$ such that $y_0>0$, $h(t;\vec{y}) \equiv 0$ for all $t\geq 0$;
\item for all $\vec{y}\in \YY^N$, $h(0;\vec{y}) = h_0(\vec{y})$.
\end{enumerate}

\begin{corollary}\label{c.truegen}
For any fixed $\vec{x}\in \XX^N$,
$$h(t;\vec{y}):= \EE^{\vec{x}}\big[H(x(t),y) \big]=\EE^{\vec{x}}\Big[\prod_{i=0}^{N} q^{y_i(x_i+i)}\Big]$$
is the unique solution to the general parameter version of the true evolution equation with initial data
$h_0(\vec{y}):= \prod_{i=0}^{N} q^{y_i(x_i+i)}$.
\end{corollary}
\begin{proof}
It is clear from Theorem \ref{t.genintertwine} that $h(t;\vec{y})$ solves the general parameter version of the true evolution equation with initial data as given. To show the uniqueness of solutions to the true evolution equation observe first that the matrix $P_{t+1}^{{\rm Boson}}$ splits into a direct sum of matrices acting only on the subspace $\YY^N_k$, over $k\geq 0$. Each of these matrices is finite and triangular in the sense that $\vec{y}\in \YY^N_k$, $h(t+1;\vec{y})$ depends only upon the values of $h(t;\vec{y}')$ for those $\vec{y}'$ such that for all $i\in \{0,\ldots, N\}$, $y_i'+\cdots + y_N' \leq y_i+\cdots +y_N$. This clearly implies uniqueness (and existence) of solutions.
\end{proof}

Theorem \ref{t.genintertwine} also implies a form of Markov duality, though this requires employing a time reversal of one of the chains. Since this duality is not utilized, we forego stating it.

\subsection{Equivalence of true and free evolution equations}\label{s.bethegen}
The following theorem is a general parameter analog of Proposition \ref{p.freeevo}. The proof is essentially given in \cite[Section 3]{Pov} (though that is stated for all $a_i\equiv 1$) and relies upon (\ref{e.binexp}) a three parameter generalization of the Binomial expansion.

\begin{proposition}\label{p.freeevogen}
Fix $|q|<1$, $\nu\in [0,1)$, and integers $N,k\geq 1$. Further, fix site/particle dependent jump parameters $a_i>0$ for all $i\in \{1,\ldots, N\}$, and time dependent jump parameters $\mu_t\in[\nu,1)$ for all $t\in \Z_{>0}$. If $u:\R_{\geq 0} \times \Z^k \to \C$ solves:
\begin{enumerate}
\item {\em ($k$-particle free evolution equation)} for all $\vec{n}\in \Z^k$ and $t\geq 0$
$$
u(t+1;\vec{n}) = \prod_{i=1}^{k} \big[\nabla_{a_{n_i}\mu_{t+1},\nu}\big]_i u(t;\vec{n})
$$
where $[\nabla_{\mu,\nu}]_i u(t;\vec{n}) := \frac{\mu-\nu}{1-\nu} u(t;\vec{n}_{i}^{-}) + \frac{1-\mu}{1-\nu} u(t;\vec{n})$;
\item {\em ($k-1$ two-body boundary conditions)} for all $\vec{n}\in \Z^k$ such that for some $i\in \{1,\ldots, k-1\}$, $n_i= n_{i+1}$ and all $t\geq 0$,
$$
\alpha u(t; n_{i,i+1}^{-}) +\beta u(t;\vec{n}_{i+1}^{-}) + \gamma u(t;\vec{n}) -u(t;\vec{n}_i^{-}) = 0
$$
where the parameters $\alpha,\beta,\gamma$ are defined in terms of $q$ and $\nu$ as
$$
\alpha = \frac{\nu(1-q)}{1-q\nu},\qquad \beta= \frac{q-\nu}{1-q\nu},\qquad \gamma = \frac{1-q}{1-q\nu};
$$
\item {\em (initial data)} for all $\vec{n}\in \Weyl^N_k$, $u(0;\vec{n}) = h_0\big(\vec{y}(\vec{n})\big)$;
\end{enumerate}
then for all $\vec{n}\in \Weyl_k^N$, and all $t\geq 0$,
$
u(t;\vec{n}) = h\big(t;\vec{y}(\vec{n})\big)
$
where $h(t;\vec{y})$ is the solution to the general parameter version of the true evolution equation with initial data $h_0$.
\end{proposition}
\begin{proof}
This result is essentially contained in \cite[Section 3]{Pov}. In order to prove it, it suffices to show $\tilde{h}(t;\vec{y}):= u\big(t;\vec{n}(\vec{y})\big)$ solves the general parameter version of the true evolution equation with initial data $h_0$, when restricted to $\vec{y}\in \YY^N_k$. When $k=1$ this is quite evident, though when $k>1$ it becomes necessary to utilize the two-body boundary conditions. This is because the free evolution equation expresses $u(t+1;\vec{n})$ for $\vec{n}\in \Weyl^N_k$ in terms of $u(t;\vec{n}')$ with some $\vec{n}'$ which may not be in $\Weyl^N_k$. Specifically, this happens when there are clusters of equal coordinates in $\vec{n}$, in which case the boundary condition satisfied by $u$ enables us to reexpress $u(t+1;\vec{n})$ in terms of $u(t;\vec{n}')$ with all $\vec{n}'\in \Weyl^N_k$. This is facilitated by the following result.
\begin{lemma}\label{l.povred}
If a function $f:\Z^m\to \R$ satisfies the boundary conditions
$$
\alpha u(t; n_{i,i+1}^{-}) +\beta u(t;\vec{n}_{i+1}^{-}) + \gamma u(t;\vec{n}) -u(t;\vec{n}_i^{-}) = 0,
$$
for all $\vec{n}\in \Z^k$ such that for some $i\in \{1,\ldots, m-1\}$, $n_i=n_{i+1}$, then it also satisfies
$$\prod_{i=1}^{m} \big[\nabla_{\mu,\nu}\big]_i f(n,\ldots, n) = \sum_{j=0}^{m} \varphi_{q,\mu,\nu}(j|m) f(\underbrace{n,\ldots,n}_{m-j},\underbrace{n-1,\ldots, n-1}_{j}).$$
\end{lemma}
\begin{proof}
This is shown in Section 3.1 of \cite{Pov} via the following generalization of the Binomial expansion. Consider an associative algebra generated by $A,B$ obeying the quadratic homogeneous relation
$$
BA= \alpha AA + \beta AB + \gamma BB.
$$
Then for $p = \frac{\mu-\nu}{1-\nu}$,
\begin{equation}\label{e.binexp}
\big(p A + (1-p)B\big)^m = \sum_{j=0}^{m}  \varphi_{q,\mu,\nu}(j|m) A^j B^{m-j}.
\end{equation}
\end{proof}
Lemma \ref{l.povred} may be applied to each cluster of equal elements in $\vec{n}$, starting with the cluster including $n_k$ and ending with the cluster including $n_1$. The repeated application of the lemma shows that
$$
u(t+1;\vec{n})= \prod_{i=1}^{k} \big[\nabla_{a_{n_i}\mu_{t+1},\nu}\big]_i u(t;\vec{n}) = \big[\mathcal{A}_{q,a_1\mu_{t+1},\nu}\big]_1\cdots \big[\mathcal{A}_{q,a_N\mu_{t+1},\nu}\big]_N \tilde{h}(t;\vec{y})=P_t^{{\rm Boson}}\tilde{h}(t;\vec{y}).
$$
Since the initial data matches, this shows (by the uniqueness of solutions to the true evolution equation) that $\tilde{h}(t;\vec{y}) = h(t;\vec{y})$
\end{proof}

\subsection{Nested contour integral and Fredholm determinant formulas}\label{s.nogennest}

Theorem \ref{t.moments} has a straightforward analog for the general parameter version of the $N$-particle $(q,\mu,\nu)$-TASEP if all $a_i\equiv 1$ (though the $\mu_t\in [\nu,1)$ by vary). The only difference in the nested contour integral formula is the replacement
$$
\left(\frac{1-\mu z_j}{1-\nu z_j}\right)^t \,\mapsto\, \prod_{s=1}^{t} \frac{1-\mu_s z_j}{1-\nu z_j}.
$$
Using this minor modification, Theorem \ref{t.fred} can likewise be modified by changing the definition of the function $g(w)$ by the replacement
$$
\left(\frac{(\mu w;q)_{\infty}}{(\nu w;q)_{\infty}}\right)^t \, \mapsto \prod_{s=1}^{t} \frac{(\mu_s w;q)_{\infty}}{(\nu w;q)_{\infty}}.
$$

It is not clear how to construct a similar sort of nested contour integral solution when the $a_i$ are not all equal. A similar difficulty arises in the context of the ASEP with bond dependent jump parameters where \cite{BCS} shows that duality and a reduction of the associated true evolution equation to free evolution equation with boundary conditions holds, yet there is no clear nested contour integral formulas for moments.

However, if the parameter $\nu=0$, then \cite[Theorem 2.1(2)]{BCdiscrete} did find general $a_i$ and $\mu_t$ parameter nested contour integral solutions given by the following replacement of terms in the integrand of right-hand side of (\ref{lhsthm}):
$$
\prod_{j=1}^{k} \left(\frac{1-\nu z_j}{1-z_j}\right)^{n_j} \left(\frac{1-\mu z_j}{1-\nu z_j}\right)^t \frac{dz_j}{z_j(1-\nu z_j)} \mapsto
\prod_{j=1}^{k} \prod_{i=1}^{n_j}\frac{a_i}{a_i-z_j}\prod_{s=1}^{t} (1-\mu_s z_j) \frac{dz_j}{z_j}.
$$
See Section \ref{s.mac} for figure discussion.

\subsection{Proof of the $(q,\mu,\nu)$-deformed Binomial distribution identity}\label{s.binidentityproof}

The proof of Proposition \ref{p.dualitymatch} which we now present is a modification of that of \cite[Lemma 3.7]{BCdiscrete}. That case corresponds with setting $\nu=0$. This leads to a simplification since the desired equality (\ref{e.toprovelhs}) within the below proof becomes
$$
\sum_{r=0}^{\ell} (-1)^r q^{\frac{r(r-1)}{2}} \frac{(q;q)_{y}}{(q;q)_r (q;q)_{y-r}} = \begin{cases} 1& \textrm{if } \ell=0\\ 0 &\textrm{otherwise}.\end{cases}
$$
That identify is \cite[Corollary 10.2.2(c)]{AAR}. The present proof requires further manipulations and ultimately appeals to the more general identity in (\ref{e.qgauss}) of Heine's 1847 $q$-generalization of Gauss' summation formula for the $q$-hypergeometric series $_2\phi_1$.

\begin{proof}[Proof of Proposition \ref{p.dualitymatch}]
We prove the first identity of the proposition, as the second one follows similarly (or through taking $m\to +\infty$). Define
\begin{equation}\label{e.sm}
S_{m,y} = \sum_{j=0}^{m} \varphi_{q,\mu,\nu}(j|m) q^{jy}.
\end{equation}
To prove the lemma we must show that $S_{m,y} = S_{y,m}$ for all $y,m$. In order to show this we will first prove that for each $m\geq 0$ there exists a lower triangular matrix $T^m = \{T^m_{i,j}:i,j\geq 0\}$ such that
\begin{equation}\label{e.tm}
T^m \big(S_{m,0},S_{m,1},\ldots\big)^\top = (1,1,\ldots)^\top.
\end{equation}
Since $T^m$ is lower triangular, this relation uniquely characterizes the elements of $S_{m,\cdot}$. Therefore, to show that $S_{m,y}=S_{y,m}$ it suffices to prove that
\begin{equation}\label{e.tmswitch}
T^m \big(S_{0,m},S_{1,m},\ldots\big)^\top = (1,1,\ldots)^\top.
\end{equation}

We start by finding $T^m$ so that (\ref{e.tm}) holds. Applying Lemma \ref{l.normalized} with $\mu,\nu$ replaced by $q^y\mu,q^y\nu$ we find that
\begin{equation}\label{e.yexp}
\frac{(\nu;q)_y}{(\mu;q)_y\, (\nu q^m;q)_y}\, \sum_{j=0}^{m}  \varphi_{q,\mu,\nu}(j|m)\,q^{yj}\, (\mu q^{m-j};q)_y = 1.
\end{equation}
In the above expression we have used that
$$
\varphi_{q,q^y \mu,q^y\nu}(j|m) = \varphi_{q,\mu,\nu}(j|m) q^{yj} \frac{(\nu;q)_y\, (\mu q^{m-j};q)_y}{(\mu;q)_y\, (\nu q^m;q)_y}.
$$
We may use the expansion
\begin{equation}\label{e.expandq}
(a;q)_y = \sum_{r=0}^{y} (-a)^r q^{\frac{r(r-1)}{2}} \frac{(q;q)_y}{(q;q)_r\, (q;q)_{y-r}}
\end{equation}
to rewrite (\ref{e.yexp}) as the identity
$$
\frac{(\nu;q)_y}{(\mu;q)_y\, (\nu q^m;q)_y}\, \sum_{r=0}^{y} (-\mu)^r q^{mr} q^{\frac{r(r-1)}{2}} \frac{(q;q)_y}{(q;q)_r\, (q;q)_{y-r}} S_{m,y-r}= 1,
$$
where $S_{m,y-r}$ is defined in (\ref{e.sm}). This identity can be rewritten in the form of (\ref{e.tm}) with $T^m$ read off from the above expression. In order to prove (\ref{e.tmswitch}) we must prove that
\begin{equation}\label{e.toprove}
\frac{(\nu;q)_y}{(\mu;q)_y\, (\nu q^m;q)_y}\, \sum_{r=0}^{y} (-\mu)^r q^{mr} q^{\frac{r(r-1)}{2}} \frac{(q;q)_y}{(q;q)_r\, (q;q)_{y-r}} S_{y-r,m}= 1
\end{equation}
holds ($S_{m,y-r}$ has been replaced by $S_{y-r,m}$).

The rest of this proof is devoted to showing (\ref{e.toprove}). After some minor manipulations to this desired identity, we will reduce it to an application of Heine's $q$-Gauss summation formula. We may use the expansion (\ref{e.expandq}) rewrite (\ref{e.toprove}) as
\begin{equation}\label{e.toprovenext}
\frac{(\nu;q)_y}{(\mu;q)_y}\, \sum_{r=0}^{y} (-\mu)^r q^{mr} q^{\frac{r(r-1)}{2}} \frac{(q;q)_y}{(q;q)_r\, (q;q)_{y-r}} S_{y-r,m}=  \sum_{r=0}^{y} (-\nu q^m)^r q^{\frac{r(r-1)}{2}} \frac{(q;q)_y}{(q;q)_r\, (q;q)_{y-r}}.
\end{equation}
We prove (\ref{e.toprovenext}) by matching coefficients of powers of $q^m$ on both sides of the desired identity. Expanding both sides of (\ref{e.toprovenext}) we can gather the coefficients of $(q^m)^\ell$ for $\ell=0,\ldots, y$. In order that all of the coefficients match, we must show that for all $\ell=0,\ldots, y$ we have
$$
\frac{(\nu;q)_y}{(\mu;q)_y}\, \sum_{r=0}^{\ell} (-\mu)^r q^{\frac{r(r-1)}{2}} \frac{(q;q)_y}{(q;q)_r\, (q;q)_{y-r}} \varphi_{q,\mu,nu}(\ell-r|y-r)=  (-\nu)^\ell q^{\frac{\ell(\ell-1)}{2}} \frac{(q;q)_y}{(q;q)_\ell\, (q;q)_{y-\ell}}.
$$
From the definition of $\varphi$ this is equivalent to showing that for all $\ell=0,\ldots, y$ we have
\begin{equation}\label{e.toprovelhs}
\sum_{r=0}^{\ell} (-1)^r q^{\frac{r(r-1)}{2}} \frac{(q;q)_y}{(q;q)_r\, (q;q)_{y-r}} \frac{(\nu/\mu;q)_{\ell-r}}{(\nu;q)_{y-r}} = (-\nu/\mu)^\ell q^{\frac{\ell(\ell-1)}{2}}  \frac{(\mu;q)_y}{(\nu;q)_y\, (\mu;q)_{y-\ell}}.
\end{equation}
Using the identity (\ref{e.gasp}(C))
we can rewrite the left-hand side of (\ref{e.toprovelhs}) as
\begin{align}
\textrm{LHS(\ref{e.toprovelhs})}&= \frac{(\nu/\mu;q)_{\ell}}{(\nu;q)_y} \sum_{r=0}^{\ell} \frac{(q^{-\ell};q)_r \, (\nu^{-1} q^{1-y};q)_{r}}{(q;q)_r\, (\mu \nu^{-1} q^{1-\ell};q)_r} (\mu q^y)^r\\
&=\frac{(\nu/\mu;q)_{\ell}}{(\nu;q)_y}  {_2\phi_1} (q^{-\ell},\nu^{-1} q^{1-y};\mu\nu^{-1} q^{1-\ell};q;\mu q^y).
\end{align}
The second equality follows from (\ref{d.hypphi}) and the fact that having $q^{-\ell}$ as an argument cuts the infinite summation in ${_2\phi_1}$ off for all $r>\ell$. Using this ${_2\phi_1}$ expression for the left-hand side of (\ref{e.toprovelhs}) we can rewrite (\ref{e.toprovelhs}) as
\begin{equation*}
{_2\phi_1} (q^{-\ell},\nu^{-1} q^{1-y};\mu\nu^{-1} q^{1-\ell};q;\mu q^y) = (-\nu/\mu)^\ell q^{\frac{\ell(\ell-1)}{2}} \frac{(\mu;q)_y}{(\nu;q)_y\, (\mu;q)_{y-\ell}}.
\end{equation*}
We must show this identity for all $\ell=0,\ldots, y$. Applinyg (\ref{e.qgauss}) with $a= q^{-\ell}$, $b = \nu^{-1} q^{1-y}$, and $c=\mu\nu^{-1} q^{1-\ell}$ reduces this desired identity to
$$
\frac{(\nu/\mu;q)_\ell}{(\nu;q)_y}\, \frac{(\mu \nu^{-1} q;q)_{\infty}\, (\mu q^{y-\ell};q)_{\infty}}{(\mu \nu^{-1} q^{1-\ell};q)_{\infty}\, (\mu q^y;q)_{\infty}} = (-\nu/\mu)^\ell q^{\frac{\ell(\ell-1)}{2}} \frac{(\mu;q)_y}{(\nu;q)_y\, (\mu;q)_{y-\ell}}
$$
which is easily confirmed by using the identities (\ref{e.gasp}(A)) and (\ref{e.gasp}(C)). This proves the identity (\ref{e.toprove}) and hence completes the proof of Proposition \ref{p.dualitymatch}.
\end{proof}

\section{Discussion of results, extensions, open problems and relation to literature}\label{s.discuss}

Without going into too much detail, we provide some discussion below with a focus towards possible extensions and new directions of research related to this work. We do not attempt a full survey the literature related to the present work or to these extension and new directions.

\subsection{Stationary version of the  $(q,\mu,\nu)$-Boson process and $(q,\mu,\nu)$-TASEP}\label{s.stat}

Evans-Majumdar-Zia \cite{Evans} characterized the jump distributions for spatially homogeneous discrete time zero range processes (called {\it mass transport models} in \cite{Evans} or {\it zero range chipping models} in \cite{Pov}) on  periodic domains which have factorized steady states (or in other words, invariant measures which are product measures). This class of processes involve moving $s_i$ out of $y_i(t)$ particles from site $i$ to $i-1$ at time $t+1$ (independently and in parallel for all $i$) according to a jump distribution $\varphi(s_i|y_i(t))$. Povolotsky \cite{Pov} sought to characterize those jump distributions $\varphi$ which additionally led to processes solvable via Bethe ansatz (see Sections \ref{s.betheant} and \ref{s.plan}) and found that $\varphi_{q,\mu,\nu}$ constitutes that set.

This paper primarily focuses on the $N$-site $(q,\mu,\nu)$-Boson process in which there are no invariant measures (eventually all particles move to site $0$). For the moment consider the model on $\Z$ with state space $(\Z_{\geq 0})^\Z$, so that a state $\vec{y} = \{y_i\}_{i\in \Z}$. Taking an infinite volume analog of the factorized steady states from \cite{Evans} we arrive at a class of product measures on $\vec{y}$ indexed by a parameter $\rho\in [0,1)$ in which for each $i\in \Z$,
\begin{equation}\label{e.pyi}
\PP(y_i=n) = \rho^n\, \frac{(\nu;q)_n}{(q;q)_n} \, \frac{(\rho;q)_{\infty}}{(\rho\nu;q)_\infty}.
\end{equation}
We speculate that these constitute invariant measures for the $(q,\mu,\nu)$-Boson process on $\Z$ (we do not confirm that here). It would be interesting to classify the full set of translation invariant stationary measures for this process.

The continuous time counterparts of the mass transport models in \cite{Evans} are the totally asymmetric zero range processes (TAZRPs) in which a single particle moves from site $i$ to $i-1$ with rate $g(y_i)$, independently and in parallel over all $i\in \Z$. The rate function $g$ plays an analogous role to the jump probability distribution, though in continuous time only one particle can jump (as opposed to clusters which can move in the discrete time models). Under very mild growth conditions on $g$, it is known that TAZRPs have invariant measures which are product measures \cite[Proposition 3.3.11]{BorCor} (with one point distribution related to the rate $g$). This should be compared to the discrete time processes in which very particular conditions on the jump distribution must be satisfied, as shown in \cite{Evans}.

Bal\'{a}zs-Komj\'{a}thy-Sepp\"{a}l\"{a}inen \cite{BKS} used {\it second class particle} and {\it coupling methods} to prove that a wide class of TAZRPs demonstrate cube root fluctuations (i.e. $t^{1/3}$) in their particle current through {\it characteristics}. This cube root behavior is an indication of membership in the KPZ universality (see the review \cite{ICReview}). It would be interesting to develop the methods used in \cite{BKS} to this discrete time setting and prove cube root fluctuations in this manner. Note, it may only be possible to implement this approach in the case of product form invariant measures.

%The work of \cite{Evans,Pov} dealt only with jump distributions which were spatially homogeneous. Section \ref{s.genparsection} discusses some developments when the jump distribution parameters (in particular the $\mu$) change with each site. It would be interesting to see whether the same product measure given by is still invariant for the general parameter $(q,\mu,\nu)$-Boson process.

\subsection{Plancherel theory and coordinate Bethe ansatz}\label{s.plan}

Utilizing the coordinate Bethe ansatz, Povolotsky \cite{Pov} constructed eigenfunctions for $k$-particle restriction of the $(q,\mu,\nu)$-Boson process transition matrix on a periodic domain and on $\Z$. Let us focus on the case of $\Z$. In this case, the eigenfunctions are indexed by $k$ complex numbers (sometimes called quasi-momenta). In order to solve the true evolution equation for a specific space of initial data, it is necessary to determine which subset of the eigenfunctions constitute a complete basis for the desired space and how these eigenfunctions should be normalized in such a decomposition. This problem goes under the general name of {\it completeness of the coordinate Bethe ansatz} and is achieved by proving a {\it Plancherel theory}. For the case of the $(q,\mu,\nu)$-Boson process with $\nu=0$ this has been achieved in \cite{BCPS}. It would be interesting to develop the analogous theory for general $\nu\neq 0$. Note that \cite[Conjecture 2]{Pov} provides a conjecture (which agrees with the $\nu=0$ case proved in \cite{BCPS}) for a portion of the desired results. One output of an analogous Plancherel theory to that of \cite{BCPS} would be a systematic and direct route to solve the $(q,\mu,\nu)$-Boson process true evolution equation for more general initial data.

\subsection{Algebraic Bethe ansatz}\label{s.aba}

The continuous time $q$-Boson process (see Section \ref{s.limitproc}) was introduced by Sasamoto-Wadati \cite{SasWad} in the language of the algebraic Bethe ansatz. The generator for the process arises from a certain representation of the $q$-Boson Hamiltonian, which is built (in a standard way) from quantum $L$ and $R$ matrices involving the $q$-Boson algebra. In principal the coordinate eigenfunctions produced in \cite{Pov} should be accessible from the algebraic Bethe ansatz (though this mapping of the algebraic to coordinate eigenfunctions have not been performed). The coordinate eigenfunctions of the $(q,\mu,\nu)$-Boson process are different than those for the $q$-Boson process (they depend non-trivially on $\nu$). This suggests that if one could fit the discrete time $(q,\mu,\nu)$-Boson process into the algebraic Bethe ansatz it would require use of a modified $L$ matrix. When $\nu=0$ the coordinate eigenfunctions  for the $(q,\mu,0)$-Boson process match those of the $q$-Boson process. However, even with $\nu=0$, it has not yet been determined how the $(q,\mu,0)$-Boson process arises from the algebraic Bethe ansatz transfer matrix. For a further discussion on this, see \cite[Section 1.2.4]{BCPS}.

\subsection{Measures on interlacing partitions and symmetric functions}\label{s.mac}

The nested contour integral moment formulas of Theorem \ref{t.moments} are reminiscent of formulas which arise in the theory of Macdonald processes \cite{BorCor}. When $\nu=0$, \cite[Section 6]{BCdiscrete} makes a very clear link between these formulas and Macdonald processes (as well as between the $q$-Boson process and commutation relations involving Macdonald first difference operators). This $\nu=0$ link is facilitated by the fact that there exist nested contour integral formulas for moments of the general $a_i$ parameter $(q;\mu;0)$-Boson process considered in Section \ref{s.genparsection}. As observed in Section \ref{s.nogennest}, it is not clear how to produce such formulas when $\nu\neq 0$ (or whether such formulas exist). Hence, it remains unclear whether there exists an analogous theory to that of Macdonald processes which relates exactly to the $(q,\mu,\nu)$-Boson process.

Also in the $\nu=0$ context, \cite{BCdiscrete} introduced another discrete time variant of $q$-TASEP (different than the $(q,\mu,0)$-TASEP and related to the so called  $\beta$ specialization of Macdonald processes). This process, called Bernoulli $q$-TASEP, was studied in \cite{BCdiscrete} via the same sort duality approach utilized herein. It is not clear whether there is a $\nu\neq 0$ generalization of this Bernoulli $q$-TASEP which is similarly solvable.

\subsection{Limits of the $(q,\mu,\nu)$-Boson process and $(q,\mu,\nu)$-TASEP}\label{s.limitproc}

We have already alluded to the $\nu=0$ degeneration of the $(q,\mu,\nu)$-Boson process and $(q,\mu,\nu)$-TASEP. This process coincides with the discrete time geometric $q$-TASEP (and the $q$-Boson process related through the ASEP-ZRP and particle-hole transform) which was studied in \cite{BCdiscrete}. The results of this paper generalize some of those in \cite{BCdiscrete} to this $\nu\neq 0$ setting. If we further set $\mu=(1-q)\e$ and scale time like $\e^{-1}$, then as $\e\to 0$ the discrete time geometric $q$-TASEP converges to the continuous time Poisson $q$-TASEP of \cite{BorCor,BCS}. A further limit (cf. \cite[Section 6]{BCS}) involving $q\to 1$ yields the O'Connell-Yor semi-discrete directed polymer model and a yet further limit (cf. \cite{MRQ} or \cite[Section 3]{BCF}) yields the Kardar-Parisi-Zhang equation.

Povolotsky \cite{Pov} describes quite a few other degenerations of the $(q,\mu,\nu)$-Boson process and $(q,\mu,\nu)$-TASEP including: a TASEP with generalized update, a continuous time fragmentation model, a multiparticle hopping asymmetric diffusion process which interpolated between TASEP and the drop-push model, a discrete time zero range process involving at most one particle jumping per-site, and the Asymmetric Avalanche Process -- see \cite{Pov} for the relevant degenerations, descriptions and references for these processes. It would be interesting to likewise degenerate Theorems \ref{t.moments} and \ref{t.fred} to these models and consequently study their long-time and large-scale behavior. For continuous time Poisson $q$-TASEP, the O'Connell-Yor semi-discrete directed polymer model and the Kardar-Parisi-Zhang equation this has been done in \cite{BorCor,BCF,BCFV, FV}.

Let us briefly illustrate one of the other degenerations mentioned above. The multiparticle hopping asymmetric diffusion process arises by setting $\mu=q\in (0,1)$, $\nu= \frac{q-\e}{1-\e}$ and rescaling time by a factor of $\e^{-1}$ (i.e. let $t= \e^{-1} \tau$ where $\tau$ will represent a continuous time parameter in the $\e\to 0$ limit). Straightforward asymptotics reveals that $\varphi_{q,\mu,\nu}(j|m) = \e\,\,\frac{1}{[j]_{q^{-1}}} + O(\e^2)$ where $[j]_{q^{-1}} = \frac{1-q^{-j}}{1-q^{-1}}$. (Note: \cite{Pov} contains a small mistake in the rate, which would correspond in the present notation to replacing $j$ by $m$.) Hence, the continuous time limit as $\e\to 0$ of the $(q,\mu,\nu)$-Boson process under this scaling is as follows: for each site $i$ and each $j\in \{1,\ldots, y_i(\tau)\}$ there is an exponential alarm clock which rings at rate $\frac{1}{[j]_{q^{-1}}}$. When this occurs, $j$ particles are moved from site $i$ to site $i-1$ and the alarm is reset. All alarms are independently distributed. The corresponding limit of $(q,\mu,\nu)$-TASEP has particle $x_n(\tau)$ jumping to site $x_n(\tau) + j$ for $j\in \{1,\ldots,x_{n-1}(\tau)-x_{n}(\tau)-1\}$ according to an exponential alarm with the rate $\frac{1}{[j]_{q^{-1}}}$.

Theorems \ref{t.moments} and \ref{t.fred} both have clear degenerations for this later model. Recall that we are presently considering the scaling $\mu=q$, $\nu= \frac{q-\e}{1-\e}$ and $t=\e^{-1} \tau$. Focusing on Theorem \ref{t.fred}, the only term which requires a bit of care as $\e\to 0$ is (in the definition of $g(w)$)
$$
\left(\frac{(\mu w;q)_{\infty}}{(\nu q;q)_{\infty}}\right)^{t} \to e^{\tau \sum_{i=0}^{\infty}\frac{q^i w}{1-q^{i+1} w}}.
$$
Thus Theorem \ref{t.fred} holds for the multiparticle hopping asymmetric diffusion process with
$$
g(w) = \left(\frac{1}{1-w}\right)^n e^{\tau \sum_{i=0}^{\infty}\frac{q^i w}{1-q^{i+1} w}} \frac{1}{(qw;q)_{\infty}}.
$$
We do not pursue further asymptotics of this process here.

\end{document}